\documentclass[a4paper,11pt,reqno]{amsart}
\usepackage{latexsym,amscd,amssymb,url}
\usepackage{comment}
\usepackage{extarrows} 
\usepackage{mathtools}
\usepackage{tikz-cd}

\usepackage[all,cmtip]{xy}  
\usepackage{graphicx} 
\usepackage{xcolor} 
\usepackage{enumitem} 
\definecolor{dblue}{rgb}{0,0,.6}
\usepackage[colorlinks=true, linkcolor=dblue, citecolor=dblue, filecolor = dblue, menucolor = dblue, urlcolor = dblue]{hyperref}

\numberwithin{equation}{section}

\newtheorem{theorem}{Theorem}[section]

\theoremstyle{plain}

\newtheorem{corollary}[theorem]{Corollary}

\newtheorem{lemma}[theorem]{Lemma}

\newtheorem{proposition}[theorem]{Proposition}

\newtheorem{remark}[theorem]{Remark}

\newcommand{\Z}{\mathbb Z}
\newcommand{\Q}{\mathbb Q}
\newcommand{\A}{\mathbb A}

\newcommand{\C}{\mathbb C}

\newcommand{\R}{\mathbb R}

\newcommand{\CP}{\mathbb P}

\newcommand{\F}{\mathbb F}

\newcommand{\im}{\operatorname{im}}

\newcommand{\Pic}{\operatorname{Pic}}

\newcommand{\Aut}{\operatorname{Aut}}

\newcommand{\id}{\operatorname{id}}

\newcommand{\Spec}{\operatorname{Spec}}

\newcommand{\Gal}{\operatorname{Gal}}

\newcommand{\NS}{\operatorname{NS}}

\newcommand{\CH}{\operatorname{CH}}

\newcommand{\cl}{\operatorname{cl}}

\newcommand{\proet}{\text{pro\'et}}

\newcommand{\reg}{\operatorname{reg}}

\newcommand{\Frob}{\operatorname{Frob}}

\newcommand{\del}{\partial}

\newcommand{\dashedlongrightarrow}{\xymatrix@1@=15pt{\ar@{-->}[r]&}}
\renewcommand{\longrightarrow}{\xymatrix@1@=15pt{\ar[r]&}}
\renewcommand{\mapsto}{\xymatrix@1@=15pt{\ar@{|->}[r]&}}
\renewcommand{\twoheadrightarrow}{\xymatrix@1@=15pt{\ar@{->>}[r]&}}
\newcommand{\hooklongrightarrow}{\xymatrix@1@=15pt{\ar@{^(->}[r]&}}
\newcommand{\congpf}{\xymatrix@1@=15pt{\ar[r]^-\sim&}}
\renewcommand{\cong}{\simeq}

\begin{document}    


\title[Remarks on Milnor K-theory and Tate's conjecture for divisors]{Remarks on Milnor K-theory and Tate's conjecture for divisors}



\author{Stefan Schreieder} 
\address{Institute of Algebraic Geometry, Leibniz University Hannover, Welfengarten 1, 30167 Hannover, Germany.}
\email{schreieder@math.uni-hannover.de}

\date{November 23, 2024} 
%



\begin{abstract}   
 We show that the Tate conjecture for divisors over a finite field $\F$ is equivalent to an explicit algebraic problem about the third Milnor K-group of the function field $\bar \F(x,y,z)$ in three variables over $\bar \F$. 
\end{abstract}
 
\maketitle

\section{Introduction}

An  important invariant of  a field $L$ is the Milnor K-theory
$
K^M_\ast (L) 
$, see \cite{milnor}.  
This is the graded algebra obtained as the quotient of the free tensor algebra on the abelian group $L^\ast$ of units in $L$, modulo the two-sided ideal generated by $a\otimes(1-a)$ with $a\in L\setminus\{0,1\}$.  
Elements of degree $i$ are given by finite linear combinations of symbols $(g_1,\dots ,g_i)$ with $g_1,\dots ,g_i \in L^\ast$, modulo the subgroup generated by the relation 
$
(g_1,\dots ,g_i)=0
$
whenever $g_a+g_b=1$ for some $1\leq a<b\leq i$.

We will be interested in the case when $L=\bar \F(X)$ for some geometrically irreducible variety $X$  over a finite field $\F=\F_q$ with $q$ elements.
In other words, $L$ is the function field of the base change $\bar X=X\times \bar \F$ of $X$ to the algebraic closure of $\F$. 
The arithmetic Frobenius $\Frob_{q}$ relative to $\F_q$ acts via $\id\times \Frob_q$ on $\bar X$ and this induces a $\Z$-linear action
$ \bar \F(X)^\ast \to \bar \F(X)^\ast$, given explicitly by raising the coefficients of a given rational function to the $q$-th power.
This induces a $\Z$-linear action
\begin{align} \label{eq:Frobenius-on-K(P^n)}
\Frob_{q}:K^M_i (\bar \F(X))\longrightarrow K^M_i (\bar \F(X)  )
\end{align}
on the Milnor K-groups of $\bar \F(X)$,  which we denote  by $\Frob_q$. 

The  $\ell$-adic completion 
$$
K^M_i (\bar \F(X))\hat \otimes \Z_\ell   := 
\lim_{\substack{\longleftarrow\\ r}} K^M_i (\bar \F(X))\otimes \Z/{\ell^r} 
$$
is a $\Z_\ell$-module with a natural Galois action by the absolute Galois group $G_\F:=\Gal(\bar \F/\F)$.  
We then get an associated $\Q_\ell$-vector space
$$
K^M_i (\bar \F(X))\hat \otimes \Q_\ell   := (K^M_i (\bar \F(X))\hat \otimes \Z_\ell )\otimes_{\Z_\ell} \Q_\ell 
$$ 
and the Frobenius action in \eqref{eq:Frobenius-on-K(P^n)} induces a $\Q_\ell$-linear map
\begin{align} \label{eq:Frobenius-on-K(P^n)-l-adic-completion}
\Frob_{q}: K^M_i (\bar \F(X))\hat \otimes \Q_\ell \longrightarrow K^M_i (\bar \F(X))\hat \otimes \Q_\ell  .
\end{align}
We say that $\lambda\in \Q_\ell$ is an eigenvalue of this map if $\Frob_{q}-\lambda\cdot \id $ has a nontrivial kernel on $K^M_i (\bar \F(X))\hat \otimes \Q_\ell  $.
 
The first main result of this paper is as follows.
 
\begin{theorem} \label{thm:main:Sq-Milnor-K-3} 
Let $\F=\F_q$ be a finite field with $q$ elements and let $\ell$ be a prime invertible in $\F$. 
The Tate conjecture for divisors holds on all smooth projective varieties over $\F$ if and only if
$q$ is not an eigenvalue of the arithmetic Frobenius action
$$
\Frob_{q}: K^M_3 (\bar \F(\CP^3 ))\hat \otimes \Q_\ell \longrightarrow K^M_3 (\bar \F(\CP^3))\hat \otimes \Q_\ell  .
$$ 
\end{theorem}

For the Tate conjecture for divisors, we refer the reader to the survey \cite{totaro-tate} and the references therein. 
de Jong and Morrow \cite{morrow} showed  that the Tate conjecture for divisors is equivalent to the Tate conjecture for divisors on surfaces.
The above theorem shows that this is in turn entirely encoded in the third Milnor K-group of the purely transcendental extension $\bar \F(x,y,z)$ in three variables over $\bar \F$ together with its natural Frobenius action.
This may be seen as an elementary reformulation of the Tate conjecture for divisors in terms of a fairly explicit algebraic problem about the function field $\bar \F(x,y,z)$ in three variables.

Theorem \ref{thm:main:Sq-Milnor-K-3} develops further some results from joint work with Balkan \cite{Balkan-Sch}.
While the results in \cite{Balkan-Sch} are valid in arbitrary codimension, we do not know if Theorem \ref{thm:main:Sq-Milnor-K-3} admits generalizations to  cycles of higher codimension.

By Theorem \ref{thm:main:Sq-Milnor-K-3}, the Tate conjecture for divisors is equivalent to the injectivity of  $ \Frob_{q} -q\cdot \id$ on $ K^M_3 (\bar \F(\CP^3))\hat \otimes \Q_\ell  $. 
The next theorem collects some related results around this question.

\begin{theorem} \label{thm:basic-properties}
Let $\F=\F_q$ be a finite field with $q=p^m$ elements and let $\ell$ be a prime invertible in $\F$.
The following hold true:
\begin{enumerate}
\item $K^M_3 (\bar \F(\CP^3))\hat \otimes \Z_\ell \cong M\hat \otimes \Z_\ell$ for a  $G_\F$-module $M$ which is $\Z$-free of countable rank. \label{item1:thm:basic-properties}
\item The operator $\Frob_{q} -q\cdot \id$ is injective on $K^M_3 (\bar \F(\CP^3))\otimes\Q_\ell$.\label{item2:thm:basic-properties}
\item The operator $\Frob_{q} -q\cdot \id$ is injective on the $p$-adic completion $K^M_3 (\bar \F(\CP^3))\hat \otimes \Q_p$.\label{item3:thm:basic-properties}
\item There is a Galois-equivariant surjection 
$$
((\bar \F(\CP^3)^\ast)^{\otimes 3})\hat \otimes \Q_\ell\twoheadrightarrow K^M_3 (\bar \F(\CP^3))\hat \otimes \Q_\ell
$$
and the action of $\Frob_{q} -q\cdot \id$ on $((\bar \F(\CP^3)^\ast)^{\otimes 3}) \hat \otimes \Q_\ell$ is injective.\label{item4:thm:basic-properties}
\item If we set
$$
S_q:=\{\lambda\in \Z \mid \text{ there is a nonzero class $\alpha\in K^M_3(\bar \F(\CP^3))\hat \otimes \Q_\ell $ with $\Frob_q \alpha=\lambda\alpha$} \},
$$ 
then 
$S_q\subset \{\pm1,\pm q^{1/2},\pm q\}$  and $\{\pm1, \pm q^{1/2}\}\subset S_q$ for  $q=p^{4m}$.\label{item5:thm:basic-properties}
\end{enumerate}
\end{theorem}

Some comments   are in order.

Item \eqref{item1:thm:basic-properties} shows that the a priori very complicated $\Z_\ell$-module $K^M_3 (\bar \F_q(\CP^3))\hat \otimes \Z_\ell$ that is the key player in Theorem \ref{thm:main:Sq-Milnor-K-3} is in fact quite simple: it is the $\ell$-adic completion of a  $G_\F$-module $M$ whose underlying $\Z$-module is free of countable rank.
We will see in the proof (see Theorem \ref{thm:K_n-completed-countable-free} below) that $M$ is a $G_\F$-submodule of a countable direct sum $N=\bigoplus P_i$ of permutation $G_\F$-modules $P_i\cong \Z^{n_i}$.
The Frobenius eigenvalues on $N\hat \otimes \Q_\ell$ are roots of unity, see Lemma \ref{lem:permutation} below.
Item \eqref{item5:thm:basic-properties} is thus somewhat surprising, as it says that the  submodule $M\subset N$ has the property that $M\hat \otimes \Q_\ell$ has $q^{1/2}$ as eigenvalue as long as $q=p^{4m}$ is a 4-th power. 

Item \eqref{item2:thm:basic-properties}  follows from the simple observation that $\Frob_q$ acts via finite orbits on any element in $K^M_n(\bar \F(X))$.
This does of course not imply the criterion formulated in Theorem \ref{thm:main:Sq-Milnor-K-3}, because $\otimes \Q_\ell$ and $\hat \otimes \Q_\ell$ do in general not agree on non-finitely generated abelian groups. 

Item \eqref{item3:thm:basic-properties} asserts that the operator in question is injective on the $p$-adic completion, opposed to the $\ell$-adic completion needed in Theorem \ref{thm:main:Sq-Milnor-K-3}.  
In other words,  the $p$-adic version of the criterion in Theorem \ref{thm:main:Sq-Milnor-K-3} holds in fact true.
This is somewhat remarkable because the $\ell$-adic Tate conjecture for divisors is known to be equivalent to the $p$-adic version in crystalline cohomology, see \cite[Proposition 4.1]{morrow}. 

Item \eqref{item4:thm:basic-properties} shows that  $\Frob_q-q\cdot \id$ is injective on the $\ell$-adic completion $(\bar \F(\CP^3)^\ast)^{\otimes 3}\hat \otimes \Q_\ell$, which surjects onto $K^M_3 (\bar \F(\CP^3))\hat \otimes \Q_\ell$.
As before,  this does not imply the criterion  in Theorem \ref{thm:main:Sq-Milnor-K-3} because, in contrast to the case of finite-dimensional vector spaces,  an injective operator on an infinite dimensional vector space may descend to a non-injective operator on some quotient.  

Item \eqref{item5:thm:basic-properties} shows that the possible integral eigenvalues $\lambda\in \Z$ of $\Frob_q$ on $K_3^M(\bar \F(\CP^3))\hat \otimes \Z_\ell$ are $\pm 1$, $\pm q^{1/2}$, and $\pm q$.
A similar result will be proven for $K_3^M(\bar \F(Y))\hat \otimes \Q_\ell$ as long as $H^3_{nr}(\bar Y,\mu_{\ell^r}^{\otimes 3})=0$ for all $r$, see Proposition \ref{prop:weights-for-completed-K-theory} below.
This may be seen as a certain weight result for the integral eigenvalues of the Frobenius action on completed Milnor K-theory.
We do not know if the result generalizes to larger degree.

For sufficiently divisible powers $q=p^m$,  item \eqref{item5:thm:basic-properties} shows that the eigenvalues  $\pm 1$ and $\pm q^{1/2}$ actually occur.
By Theorem \ref{thm:main:Sq-Milnor-K-3}, the question whether one of the remaining eigenvalues $\pm q$ occurs is equivalent to the Tate conjecture for divisors.
(Note that if $-q$ occurs as eigenvalue of $\Frob_q$, then $q^2$ occurs for $\Frob_{q^2}$.)

\section{Preliminaries and conventions}
 \subsection{Conventions} 

For an abelian group $M$, we denote by
$$
M\hat \otimes \Z_\ell:=\lim_{\substack{\longleftarrow\\ r}} M/\ell^r\ \ \ \text{and}\ \ \ M\hat \otimes \Q_\ell=(M\hat \otimes \Z_\ell)\otimes_{\Z_\ell} \Q_\ell
$$
the $\ell$-adic completion of $M$ and its associated $\Q_\ell$-vector space, respectively. 
An element $m\in M$ is $\ell^r$-divisible (or divisible by $\ell^r$) if $m=\ell^r\cdot m'$ for some $m'\in M$; it is $\ell^\infty$-divisible if it is $\ell^r$-divisible for all $r\geq 1$.
The latter is equivalent to asking that the image of $m$ in $M\hat\otimes \Z_\ell$ vanishes.

For a field $k$ with a separable closure $k_s$, $G_k\coloneqq \Gal(k_s/k)$ denotes the absolute Galois group of $k$.  
An algebraic scheme is a separated scheme of finite type over a field. 
A variety is an integral algebraic scheme.

Throughout,  $\F=\F_q$ denotes a finite field with $q$ elements and $\bar \F$ denotes its algebraic closure.   
The arithmetic Frobenius relative to $\F=\F_q$ is denoted by $\Frob_q\in G_{\F}$.

Let $G:=\hat \Z$ be the absolute Galois group of a finite field.
We say that a $G$-module $P$ is a permutation $G$-module if $P\cong \Z^n$ admits a finite basis $\mathcal B$ and $G$ acts via permutation on $\mathcal B$. 
Since $\mathcal B$ is finite, this action must factor through a finite quotient $G\to \Z/m$. 
Consideration of the orbits of this action on $\mathcal B$ shows that $P$ decomposes into the finite sum of its indecomposable $G$-submodules, which are themselves permutation modules; moreover, $P$ is indecomposable if and only if the $G$-action on $\mathcal B$ has only one orbit and $G$ acts via cyclic permutation.

\subsection{Milnor K-groups} \label{subsec:Milnor}
If $L$ is a field, then $K^M_n(L)$ denotes the $n$-th Milnor K-group of $L$, see \cite{milnor}.
If $L=\prod_{i\in I} L_i$ is a product of finitely many fields, then we define $K^M_n(L):=\prod_{i\in I} K^M_n(L_i)$ to be the product  of the Milnor K-groups $K^M_n(L_i)$.
This agrees with the direct sum because $I$ is finite.

For instance, if $X$ is a variety over $\F$, then $\bar \F(X)$ is a product of fields (namely the function fields of the components of $\bar X=X\times \bar \F$) and so $K^M_n(\bar \F(X))$ is defined.
Similarly, if $x\in X$ denotes a schematic point, then $\bar x$ denotes the base change to $\bar \F$.
This is a finite union of points $\bar x=\{\bar x_1,\dots ,\bar x_m\}$ in $\bar X$ which form an orbit under the $G_\F$-action.  
The Milnor K-group $K^M_n(\kappa(\bar x))$ is then defined as the product  (equivalently the sum) of the Milnor K-groups $K^M_n(\kappa(\bar x_i))$ of the residue fields of the points $\bar x_i$.

If $x\in X^{(1)}$ is a codimension one point contained in the smooth locus of $X$, then there is a residue map $\del_x:K^M_n(\F(X))\to K^M_{n-1}(\kappa(x))$, see \cite{milnor}.
Moreover, with the above conventions, there is also a natural residue map $\del_{\bar x}:K^M_n(\bar \F(X))\to K^M_{n-1}(\kappa(\bar x))$.

\subsection{\'Etale cohomology}

If $X$ is an algebraic scheme over a  field $k$ and $\ell$ is a prime invertible in $k$, then we denote by $H^i(X,\mu_{\ell^r}^{\otimes n})$, $H^i(X,\Z_{\ell}(n))$, and $H^i(X,\Q_\ell(n))$ the respective continuous \'etale cohomology groups,  see \cite{jannsen}. 
If $k$ is a finite field or the algebraic closure of a finite field, then these groups agree with ordinary \'etale cohomology by \cite[(0.2)]{jannsen}, because $H^i(X,\mu_{\ell^r}^{\otimes n})$ is finite in this case and so the Mittag--Leffler condition holds true.
For convenience we denote the coefficients $\mu_{\ell^r}^{\otimes n}$ sometimes by $\Z/\ell^r(n)$.

If $X$ is a  scheme over a finite field $\F$, then we denote by $\bar X=X\times \bar \F$ the base change to an algebraic closure of $\F$.
We denote by $\ell$ a prime invertible in $\F$. 
The Galois group $G_\F$ acts on the second factor of $\bar X=X\times \bar \F$, which induces a $\Q_\ell$-linear Galois action on $ H^i(\bar X,\Q_\ell(n))$.
The Galois-invariant subspace is denoted by $ H^i(\bar X,\Q_\ell(n))^{G_\F}$.

We denote by $\Frob_q$ the endomorphism of $\bar X$ induced by $\id\times \Frob_q$; its action on \'etale cohomology is accordingly denoted by $\Frob_q^\ast$.
The analogous action on Milnor K-theory $K^M_n(\bar \F( X)) $ is induced by the $\Frob_q$-action on $\bar \F(X)^\ast$ given by raising the coefficients of a rational function to its $q$-th power.
This action will be denoted by $\Frob_q$; both actions are compatible with each other, see Theorem \ref{thm:bloch-kato} below.

\subsection{Cycle conjectures} \label{subsec:cycle-conjectures}

Let  $X$ be a smooth projective variety over a finite field $\F$ and fix a prime $\ell$ invertible in $\F$.
We say that the $1$-semi-simplicity conjecture holds in degree $i$ if the Frobenius action on $H^{2i}(\bar X,\Q_\ell(i))$ is semi-simple at the eigenvalue $1$, i.e.\ the generalized eigenspace and the eigenspace agree at the eigenvalue $1$.
We further say that the Tate conjecture holds in degree $i$ if the cycle class map
$$
\cl_X^i:\CH^i(X)\otimes_\Z \Q_\ell\longrightarrow H^{2i}(\bar X,\Q_\ell(i))^{G_\F}
$$
is surjective.
We also say that the respective conjecture holds for divisors if it holds in degree 1.

Both conjectures are trivially true in degree $0$.
The conjectures hold for divisors on abelian varieties \cite{tate-conj} and K3 surfaces   
\cite{maulik, charles,madapusipera,kim-madapusipera}, 
but are in general wide open even for divisors on smooth projective surfaces.
In fact, by work of de Jong and Morrow \cite{morrow}, the case of divisors on smooth projective varieties of arbitrary dimension is equivalent to the case of divisors on smooth projective surfaces.

We further have the following result of Milne, see \cite[Proposition 8.2 and Remark 8.5]{milne-AJM}.

\begin{lemma}[Milne] \label{lem:milne}
Let $X$ be a smooth projective variety over a finite field $\F$.
Assume that the Tate conjecture holds for divisors on $X$. 
Then the following holds:
\begin{enumerate}
\item the $1$-semi-simplicity conjecture holds in degree $1$ on $X$;
\item the Tate conjecture for divisors and the 1-semi-simplicity conjecture in degree 1 on $X$ hold for any prime $\ell'$ invertible in $\F$. 
\end{enumerate} 
\end{lemma}

\subsection{Filtrations and unramified cohomology} 
For an equi-dimensional algebraic scheme $X$ over a field $k$, we  define $X^{(j)}$ as the set of all codimension-$j$ points of $X$, i.e.\ $X^{(j)}\coloneqq\{x\in X\big|\dim X-\dim\overline{\{x\}}=j\}$.
We further denote by 
$$
F_jX\coloneqq\{x\in X\big|\dim X-\dim\overline{\{x\}}\leq j\}
$$
the set of points of codimension at most $j$.
This yields a filtration $F_0X\subset F_1X\subset \dots \subset X$ and we may regard $F_jX$ as a pro-scheme given by the system of all open subsets of $X$ that contain $X^{(j)}$, see \cite{Sch-refined}. 
For us the case $j=0$ will be particularly important and we may think about $F_0X$ as the set of generic points of $X$.

Let  $A(n)\in \{\mu_{\ell^r}^{\otimes n},\Z_\ell(n),\Q_\ell(n)\}$, where $\ell$ is a prime invertible in $k$. 
We then write
 $$
 H^i(F_jX,A(n)):=\lim_{\substack{\longrightarrow \\   U\subset X}} H^i(U,A(n))  ,
 $$
 where $U\subset X$ runs through all open subsets with $F_jX\subset U$.
If $X=Y_k$ for a variety $Y$ defined over a subfield $k'\subset k$, then any element in $\Gal(k/k')$, i.e.\ any field automorphism of $k$ that fixes $k'$, acts naturally on $ H^i(F_jX,A(n))$.
(This uses that we  may, under the given assumptions, run in the above direct limit through those open subsets that are defined over $k'$.)

The Bockstein sequence for $X$ yields in the direct limit a Bockstein sequence
$$
\dots \longrightarrow H^i(F_jX,\Z_\ell(n))\stackrel{\times \ell^r}\longrightarrow H^i(F_jX,\Z_\ell(n))\longrightarrow H^i(F_jX,\mu_{\ell^r}^{\otimes n})\longrightarrow H^{i+1}(F_jX,\Z_\ell(n))\longrightarrow \dots 
$$
that we will use, cf.\ \cite{Sch-refined}.

We denote by $N^\ast$ the coniveau filtration on $H^i(X,A(n))$, which is defined via the kernel of the natural map
$$
N^jH^i(X,A(n)):=\ker(H^i(X,A(n))\longrightarrow  H^i(F_{j-1}X,A(n)) ).
$$
We will further use the shorthand notation
$$
H^i(X,A(n))/N^j:=H^i(X,A(n))/N^jH^i(X,A(n)).
$$
 
 If $X$ is smooth, then for any codimension one point $x\in X^{(1)}$,  there is a residue map
 $$
 \del_x:H^i(F_0X,A(n))\longrightarrow H^{i-1}(x,A(n-1)),
 $$
 where $H^{i-1}(x,A(n-1))=H^{i-1}(F_0\overline{\{x\}},A(n-1))$.
 The unramified cohomology of $X$ with values in $A(n)$ is the subgroup $ H^i_{nr}(X,A(n))\subset H^i(F_0X,A(n))$ of classes that lie in the kernel of $\del_x$ for all $x\in X^{(1)}$, see \cite{CT,Sch-survey}. 
 As a consequence of purity and the localization/Gysin sequence, we have 
 $$
 H^1_{nr}(X,A(n))=H^1(X,A(n))
 $$
and
$$
H^2_{nr}(X,A(n))=H^2(X,A(n))/N^1 .
$$
Moreover, $N^1\subset H^2(X,A(n))$ agrees for $n=1$ with the subspace of algebraic classes, see \cite[\S 4.2]{CT} or \cite[Lemma 5.8 and Corollary 5.10]{Sch-refined}.

\subsection{Consequence of the Bloch--Kato conjecture}
 If $X$ is defined over $k$ and $\bar X=X\times \bar k$ denotes the base change to an algebraic (or separable) closure of $k$, then $ H^i(F_j\bar X,A(n)) $ admits a natural Galois action by the group $G_k$.
 
\begin{theorem}[Voevodsky] \label{thm:bloch-kato}
Let $X$ be a variety over a field $k$ and let $\ell$ be a prime invertible in $k$.
Then the following holds:
\begin{enumerate}
\item \label{eq:Bloch-Kato-torsion-free-0} $H^{i+1}(F_0X,\Z_{\ell}(i))$ is torsion-free;
\item there is a canonical isomorphism
\begin{align} \label{eq:Bloch-Kato-2}
 K^M_i(k(X))\hat \otimes \Z_\ell \stackrel{\cong}\longrightarrow  H^i(F_0X,\Z_{\ell}(i))\hat \otimes \Z_\ell.
\end{align}
If $X=Y_k$ for a variety $Y$ defined over a subfield $k'\subset k$, then the above isomorphism is equivariant with respect to the natural $\Aut(k/k')$-action on both sides, where $\Aut(k/k')$ denotes the group of field automorphisms of $k$ that fix $k'$.
\end{enumerate} 
\end{theorem} 
\begin{proof}
The Bloch--Kato conjecture proven by Voevodsky \cite{Voe} yields a canonical Galois-equivariant isomorphism
\begin{align} \label{eq:Bloch-Kato}
 K^M_i(k(X))/\ell^r \stackrel{\cong}\longrightarrow  H^i(F_0X,\mu_{\ell^r}^{\otimes i}) ,
\end{align}
where we used $H^i(F_0X,\mu_{\ell^r}^{\otimes i})\cong H^i(\Spec k(X),\mu_{\ell^r}^{\otimes i}) $, see \cite[p.\ 88,  III.1.16]{milne}.
Bloch noticed that this in turn implies by the Bockstein sequence that $H^{i+1}(F_0X,\Z_{\ell}(i))$ is torsion-free, see \cite[end of Lecture 5]{bloch-book} or \cite[Remark 5.14]{Sch-refined}.
This proves \eqref{eq:Bloch-Kato-torsion-free-0}.
Using this and the Bockstein sequence, we get 
a canonical Galois-equivariant isomorphism
$$
H^i(F_0X,\mu_{\ell^r}^{\otimes i})=H^i(F_0X,\Z_{\ell}(i))/\ell^r .
$$
Hence, \eqref{eq:Bloch-Kato}  induces a Galois-equivariant isomorphism on $\ell$-adic completions
\begin{align*} 
 K^M_i(k(X))\hat \otimes \Z_\ell \stackrel{\cong}\longrightarrow  H^i(F_0X,\Z_{\ell}(i))\hat \otimes \Z_\ell,
\end{align*}
as we want.
This proves item \eqref{eq:Bloch-Kato-2} and hence concludes the proof of the theorem.
\end{proof}

\begin{corollary}\label{cor:torsion-free}
Let $X$ be a variety over a field $k$ which contains all $\ell$-power roots of unity. 
Then $H^{i}(F_0X,\Z_{\ell}(j))$ is torsion-free for all $i,j\in \Z$.
\end{corollary}

\begin{proof}
If $k$ contains all $\ell$-power roots of unity, then the $\Z_\ell$-modules $H^{i}(F_0X,\Z_{\ell}(j))$ and $H^{i}(F_0X,\Z_{\ell}(i-1))$ are isomorphic (this isomorphism does of course not respect the respective Galois actions).
The corollary therefore follows directly from item \eqref{eq:Bloch-Kato-torsion-free-0} in Theorem \ref{thm:bloch-kato}. 
\end{proof}
 
  \begin{remark}
  The above theorem shows that the $\ell$-adic completion $K^M_n(k(X))\hat \otimes \Z_\ell$ is related to the $\ell$-adic cohomology of the generic point of $X$.
  In particular, if $X$ is smooth projective, then there is a natural map 
  $$
  H^n(X,\Z_\ell(n))\longrightarrow K^M_n(k(X))\hat \otimes \Z_\ell.
  $$
  This explains the idea that some cycle conjectures, resp.\ conjectures on $\ell$-adic cohomology, may have interpretations in terms of completed Milnor K-theory.
Note however that the above map is zero if $ H^n(X,\mu_{\ell^r}^{\otimes n})=N^1H^n(X,\mu_{\ell^r}^{\otimes n})$ for all $r\geq 1$.
 Remarkably,  for  $k=\bar \F$ and $n\geq 3$,  it is to the best knowledge of the author an open problem if the latter always holds,  
 see \cite[p.\ 305]{bloch-esnault}. 
It is likely that this is not the case in general, but if it holds for a given variety $X$, then no information of its $n$-th $\ell$-adic \'etale cohomology is captured by the completed Milnor K-group $K^M_n(k(X))\hat \otimes \Z_\ell$.
This explains one important subtlety of the relation between completed Milnor K-theory and \'etale cohomology.
We circumvent this problem in the present paper via the  observation that $N^1H^2(X,A(n))$ is generated by algebraic classes and so the above question is well-understood in the case of degree 2 cohomology.
  \end{remark} 
 
 \subsection{Prime to $\ell$ alterations}
 Let $X$ be a variety over a field $k$.
 An alteration of $X$ is a  a proper surjective and generically finite morphism $\tau:X'\to X$ such that $X'$ is smooth.
 Alterations exist by the work of de Jong \cite{deJong}.
 By an improvement due to Gabber \cite{IT}, we can moreover assume that $\deg(\tau)$ is coprime to any given prime $\ell$ that is invertible in $k$.
 Alterations with this property are called prime to $\ell$ alterations.
 In this case we have for instance $\tau_\ast\circ \tau^\ast=\deg(\tau)\cdot \id$  on $H^i(X,A(n))$ and on $H^i(F_0X,A(n))$, see e.g.\ \cite[Lemma 2.1]{Sch-refined}.
 For $A(n)\in \{\mu_{\ell^r}^{\otimes n},\Z_\ell(n),\Q_\ell(n)\}$, $\deg(\tau)$ will be invertible in the above groups (as it is coprime to $\ell$) and so various cohomological questions on $X$ can be checked on $X'$.

\section{The Tate conjecture and $\ell$-adically completed Milnor K-theory}

The goal of this section is to prove Theorem \ref{thm:main:Sq-Milnor-K-3}.
We start with two  lemmas.

\begin{lemma} \label{lem:lifts invariant classes integrally-i=2}
Let $X$ be a smooth projective variety over a finite field $\F$. 
Then $ H^{2}(F_0\bar X,\Z_\ell(1))^{G_\F}$ is a finitely generated $\Z_\ell$-module. 
\end{lemma}

\begin{proof}
We have the long exact sequence
$$
H^2(\bar X,\Z_\ell(1)) \stackrel{f}\longrightarrow H^{2}(F_0\bar X,\Z_\ell(1)) \stackrel{\del}\longrightarrow \bigoplus_{x\in X^{(1)}} H^1(\bar x, \Z_\ell),
$$
see \cite[Lemma 5.8 and Corollary 5.10]{Sch-refined}.
(Note that $\bar x$ is the base change of the point $x\in X^{(1)}$ to $\bar k$ and hence it is a finite union of points which form the Galois orbit under the $G_\F$-action which corresponds to $x$.)
Let now $\alpha\in  H^{2}(F_0\bar X,\Z_\ell(1))$ be $G_\F$-invariant.
Then $\del \alpha$ is $G_\F$-invariant and so it is torsion because $\bigoplus_{x\in X^{(1)}} H^1(\bar x, \Q_\ell)$ has weight at least 1  and hence contains no nontrivial Galois-invariant classes,  see e.g.\ \cite[Corollary 3.10]{Balkan-Sch}.
Since $\bigoplus_{x\in X^{(1)}} H^1(\bar x, \Z_\ell)$ is torsion-free (see Theorem \ref{thm:bloch-kato} but the case at hand is simple, see e.g.\ \cite[Lemma 5.13]{Sch-refined}), we conclude that $\del \alpha=0$ and so $\alpha$ lifts to $H^2(\bar X,\Z_\ell(1))$.
In other words,  $H^{2}(F_0\bar X,\Z_\ell(1))^{G_\F}$ is a $\Z_\ell$-submodule of the image of the map $f$ above.
Since $H^2(\bar X,\Z_\ell(1))$ is a finitely-generated $\Z_\ell$-module, so is the image of $f$.
This implies that $H^{2}(F_0\bar X,\Z_\ell(1))^{G_\F}$ is a finitely generated $\Z_\ell$-module,  as we want.
\end{proof}

\begin{lemma} \label{lem:galois invariants completion}
Let $X$ be an equi-dimensional algebraic $\F$-scheme such that  
$ 
H^{2i}(F_0\bar X,\Z_\ell(i))^{G_\F}
$ 
is a finitely generated $\Z_\ell$-module.
Then the natural map
$$
H^{2i}(F_0\bar X,\Z_\ell(i))^{G_\F}\longrightarrow \left(  H^{2i}(F_0\bar X,\Z_\ell(i))\hat \otimes \Z_\ell \right)^{G_\F}
$$
is injective.
In particular, if $(H^{2i}(F_0\bar X,\Z_\ell(i))\hat \otimes \Z_\ell)^{G_\F}=0$,  then 
$ 
H^{2i}(F_0\bar X,\Z_\ell(i))^{G_\F}=0.
$ 
\end{lemma}
\begin{proof}
Let $\alpha  \in H^{2i}(F_0\bar X,\Z_\ell(i))^{G_\F}$ be a class whose image in $H^{2i}(F_0\bar X,\Z_\ell(i))\hat \otimes \Z_\ell$ vanishes.
By the Bockstein sequence,  $\alpha$ is $\ell^{r}$-divisible for any $r\geq1$: $\alpha=\ell^r\beta_r$.
Since $H^{2i}(F_0\bar X,\Z_\ell(i))$ is torsion-free (see Corollary \ref{cor:torsion-free}), $\beta_r$ is $G_\F$-invariant for all $r$.
Hence, $\alpha$ is $\ell^r$-divisible in $H^{2i}(F_0\bar X,\Z_\ell(i))^{G_\F}$ for all $r$.
This implies $\alpha=0$, because  $H^{2i}(F_0\bar X,\Z_\ell(i))^{G_\F}$ is a finitely generated $\Z_\ell$-module by assumption.
\end{proof}

\begin{proposition} \label{prop:H^2-Gal invariant vs completion}
Let $X$ be a projective variety over a finite field $\F$.
Then
the natural map
\begin{align} \label{eq:map:prop:H^2-Gal invariant vs completion}
H^2(F_0\bar X,\Z_\ell(1))^{G_{\F}}\longrightarrow \left( H^2(F_0\bar X,\Z_\ell(1)) \hat \otimes \Z_\ell  \right)^{G_{\F}}
\end{align}
is an isomorphism.
Moreover, if the Tate conjecture holds for divisors on a prime to $\ell$ alteration $X'\to X$ of $X$,  then both groups in \eqref{eq:map:prop:H^2-Gal invariant vs completion}  vanish. 
\end{proposition}

\begin{proof}
We first aim to prove that the natural map \eqref{eq:map:prop:H^2-Gal invariant vs completion} is an isomorphism.
Using Gabber's prime to $\ell$ alterations \cite{IT}, it is straightforward to reduce to the case where $X$ is smooth projective.
Injectivity of the map under consideration then follows from Lemmas \ref{lem:lifts invariant classes integrally-i=2} and  \ref{lem:galois invariants completion}.

To prove surjectivity, 
 let $(\alpha_r)\in  H^2(F_0\bar X,\Z_\ell(1))\hat \otimes \Z_\ell$ be $G_\F$-invariant.  
 For $x\in X^{(1)}$, consider the residue
 $$
( \partial_x \alpha_r)\in H^1(\bar x,\Z_\ell)\hat \otimes \Z_\ell,
 $$
 where $\bar x$ denotes the base change of $x$ to $\bar \F$ (which may split up into a Galois orbit of points). 
The above class is Galois-invariant and we aim to show that this actually forces the class to vanish.
Applying Gabber's prime to $\ell$ alterations \cite{IT} to the closure of $x$ in $X$, we can without loss of generality assume that there is a smooth projective $\F$-variety $D$ whose function field agrees with the residue field of $x$.
 We then find that 
 $$
(\partial_x \alpha_r)\in H^1(F_0 \bar D,\Z_\ell)\hat \otimes \Z_\ell
 $$ 
 is Galois-invariant.
 Let $y\in D^{(1)}$ be a codimension one point.
 Then
 $$
\partial_y (\partial_x \alpha_r)\in H^0( \bar y,\Z_\ell(-1))\hat \otimes \Z_\ell
 $$
 is Galois-invariant and hence vanishes because the above group is a finitely generated free $\Z_\ell$-module with a Frobenius action of weight $2$.
 It follows that each $  \partial_x \alpha_r$ is unramified on $\bar D$ and hence contained in
 $$
 H^1(\bar D,\Z/\ell^r)\subset H^1(\bar x,\Z/\ell^r).
 $$
 These classes form a projective system and so they give rise to a class
 $$
 (\partial_x \alpha_r)\in H^1(\bar D,\Z_\ell)
 $$
 that is Galois-invariant.
 Any such class is torsion by the Weil conjectures proven by Deligne, see \cite{deligne-weil}.
 Since torsion classes vanish on the generic point (see Theorem \ref{thm:bloch-kato}  but the case at hand is in fact easy, see \cite[Lemma 5.13]{Sch-refined}), we find that
 $$
   \partial_x \alpha_r=0\in H^1(\bar x,\Z/\ell^r)
 $$
 for all $r$ and all $x$.
 It follows that $\alpha_r\in H^2(F_0\bar X,\Z/\ell^r(1))$ is unramified for all $r$: 
\begin{align}\label{eq:alpha_r}
  \alpha_r\in H^2_{nr}(\bar X,\Z/\ell^r(1)).
\end{align}

We claim that the natural sequence
\begin{align}\label{eq:les-NS}
0\longrightarrow \NS(\bar X)\otimes \Z/\ell^r\stackrel{f}\longrightarrow H^2(\bar X,\Z/\ell^r(1))\longrightarrow H^2_{nr}(\bar X,\Z/\ell^r(1))\longrightarrow 0  
\end{align}
is exact, where $\NS(\bar X)$ denotes the N\'eron--Severi group of $\bar X$, and $f$ is the reduction modulo $\ell$ of the inclusion $\NS(\bar X)\otimes \Z_\ell \hookrightarrow H^2(\bar X,\Z_\ell(1))$.
Apart from the injectivity of $f$ in \eqref{eq:les-NS}, all claims follow from the Gysin (resp.\ localization) sequence, see e.g.\ \cite[Lemma 5.8 and Corollary 5.10]{Sch-refined}.
To prove injectivity of $f$,  note that by \cite[Lemma 5.8 and Corollary 5.10]{Sch-refined}, we have a natural exact sequence
$$
0\longrightarrow \NS(\bar X)\otimes \Z_\ell  \longrightarrow H^2(\bar X,\Z_\ell (1))\longrightarrow H^2 (F_0 \bar X,\Z_\ell(1)) .
$$
Since $H^2 (F_0 \bar X,\Z_\ell(1)) $ is torsion-free (see Theorem \ref{thm:bloch-kato} or \cite[Lemma 5.13]{Sch-refined}), it follows that the natural map
$$
 \NS(\bar X)\otimes \Z/\ell^r  \longrightarrow H^2(\bar X,\Z_\ell (1))/\ell^r
$$
is injective.
Since
$$
 H^2(\bar X,\Z_\ell (1))/\ell^r\longrightarrow  H^2(\bar X,\Z/\ell^r (1)) 
$$
is injective by the Bockstein sequence, we conclude that $f$ is injective, as claimed.

The classes $\alpha_r$ in \eqref{eq:alpha_r} form a projective system and hence yield a class in $\lim_r H^2_{nr}(\bar X,\Z/\ell^r(1))$.
Since $\NS(\bar X)\otimes \Z/\ell^r$ is a finite group, the Mittag--Leffler condition is satisfied and so
$$
R^1\lim_{\substack{\longleftarrow}} ( \NS(\bar X)\otimes \Z/\ell^r ) =0.
$$
It follows that $(\alpha_r)\in \lim_r H^2_{nr}(\bar X,\Z/\ell^r(1))$ lifts to a class 
$$
\tilde \alpha\in  H^2(\bar X,\Z_\ell(1))=\lim_{\substack{\longleftarrow\\ r}}  H^2(\bar X,\Z/\ell^r(1)).
$$ 
By construction, the image of this class in $H^2(F_0\bar X,\Z_\ell(1))$ is Galois-invariant modulo $\ell^r$ for all $r$.
Since 
$
H^2_{nr}(\bar X,\Z_\ell(1))$ is finitely generated as a $\Z_\ell$-module, we deduce that the image of $\tilde \alpha$ in
$ 
H^2_{nr}(\bar X,\Z_\ell(1)) \subset H^2(F_0\bar X,\Z_\ell(1))
$ 
is Galois-invariant. 
The reduction modulo $\ell^r$ of this class is nothing but $\alpha_r$ and so \eqref{eq:map:prop:H^2-Gal invariant vs completion} is surjective,  as we want.
This completes the proof of the first assertion of the proposition.
 
Let us now assume that the Tate conjecture holds for a prime to $\ell$ alteration $X'\to X$ of $X$. 
Replacing $X$ by $X'$ we can without loss of generality assume that $X$ is smooth projective and the Tate conjecture holds for divisors on $X$.
By Lemma \ref{lem:milne}, the same holds for the 1-semi-simplicity conjecture.
By \cite[Theorem 6.1(1)]{Balkan-Sch}, $ H^2(F_0\bar X,\Q_\ell(1))^{G_{\F}}=0$ and so $H^2(F_0\bar X,\Z_\ell(1))^{G_{\F}}=0$ by torsion-freeness  (this follows from Hilbert 90, see Theorem \ref{thm:bloch-kato} or \cite[Lemma 5.13]{Sch-refined}). 
This concludes the proof of the proposition. 
\end{proof}

\begin{theorem} \label{thm:Tate-for-X-K_2-completed}
Let $X$ be a smooth projective variety over a finite field $\F$ and let $\ell$ be a prime invertible in $\F$.
Then the following are equivalent:
\begin{enumerate}
\item The Tate conjecture holds for divisors on $X$.\label{item:thm:Tate-for-X-K_2-completed-1}
\item The $\Frob_{q}$-action on $K^M_2(\bar \F(X))\hat \otimes \Q_\ell$ does not have $q$ as an eigenvalue.\label{item:thm:Tate-for-X-K_2-completed-2}
\item We have \label{item:thm:Tate-for-X-K_2-completed-3}
$$
\left( H^2(F_0\bar X,\Z_\ell(1)) \hat \otimes \Q_\ell  \right)^{G_{\F}}=0 .
$$
\end{enumerate}
\end{theorem}
\begin{proof}
There is a canonical Galois-equivariant isomorphism
$$
(H^2(F_0\bar X,\Z_\ell(1)) \hat \otimes \Z_\ell)\otimes_{\Z_\ell}\Z_\ell(1)  \cong H^2(F_0\bar X,\Z_\ell(2)) \hat \otimes \Z_\ell .
$$
By the Bloch--Kato conjecture in degree 2, proven by Merkurjev--Suslin  \cite{MS}, we also have a Galois-equivariant isomorphism
$$  H^2(F_0\bar X,\Z_\ell(2)) \hat \otimes \Z_\ell  \cong K^M_2(\bar \F(X))\hat \otimes \Z_\ell ,
$$  
see Theorem \ref{thm:bloch-kato}.
Altogether this proves the equivalence of \eqref{item:thm:Tate-for-X-K_2-completed-2} and \eqref{item:thm:Tate-for-X-K_2-completed-3}.

By Proposition \ref{prop:H^2-Gal invariant vs completion}, \eqref{item:thm:Tate-for-X-K_2-completed-3} implies 
$$
H^2(F_0\bar X,\Q_\ell(1))^{G_{\F}}=0 .
$$
This in turn implies item \eqref{item:thm:Tate-for-X-K_2-completed-1} by \cite[Theorem 6.1(2)]{Balkan-Sch}.
Conversely, item \eqref{item:thm:Tate-for-X-K_2-completed-1} implies \eqref{item:thm:Tate-for-X-K_2-completed-3} by  Proposition \ref{prop:H^2-Gal invariant vs completion}.
This completes the proof of the theorem.
\end{proof}

\begin{theorem}\label{thm:main:completion vs Tate for surfaces}
Let $\F$ be a finite field and let $\ell$ be a prime invertible in $\F$. 
Then the following are equivalent:
\begin{enumerate}
\item The Tate conjecture holds for smooth projective surfaces over $\F$.\label{item1:thm:main:completion vs Tate for surfaces}
\item The group $H^3(F_0 \CP_{\bar \F}^3,\Z_\ell(2)) $ has no nontrivial $G_{\F}$-invariant classes.\label{item1.5:thm:main:completion vs Tate for surfaces}
\item The group $H^3(F_0 \CP_{\bar \F}^3,\Z_\ell(2))\hat \otimes \Z_\ell$ has no nontrivial $G_{\F}$-invariant classes.\label{item2:thm:main:completion vs Tate for surfaces}
\end{enumerate}
\end{theorem}

\begin{proof}
By Theorem \ref{thm:bloch-kato}, $H^3(F_0 \CP_{\bar \F}^3,\Z_\ell(2)) $ is torsion-free and so this group has no Galois-invariant classes if and only if this holds for $H^3(F_0 \CP_{\bar \F}^3,\Q_\ell(2))$.
The equivalence of \eqref{item1:thm:main:completion vs Tate for surfaces} and \eqref{item1.5:thm:main:completion vs Tate for surfaces} follows therefore from Lemma \ref{lem:milne} and 
 \cite[Corollary 7.6]{Balkan-Sch}.

Next, assume that \eqref{item1:thm:main:completion vs Tate for surfaces} holds true and let
$$
(\alpha_r)\in H^3(F_0 \CP_{\bar \F}^3,\Z_\ell(2))\hat \otimes \Z_\ell
$$
be $G_{\F}$-invariant.
Assume that this class is nonzero.
Since $\CP^3_{\bar \F}$ is rational, $H^3(F_0 \CP_{\bar \F}^3,\Z/\ell^r(2))$ admits for $r>0$ no nontrivial unramified class.
We thus find that there is some codimension one point $x\in \CP^3 $, such that
$$
(\partial_x\alpha_r)\in  H^2(\bar x,\Z_\ell(1))\hat \otimes \Z_\ell
$$
is nonzero and $G_{\F}$-invariant.
This contradicts Proposition \ref{prop:H^2-Gal invariant vs completion} and so $( \alpha_r)$ was zero and \eqref{item2:thm:main:completion vs Tate for surfaces} holds true. 

Conversely,  assume that \eqref{item2:thm:main:completion vs Tate for surfaces} holds.
We aim to prove 
$ 
H^3(F_0 \CP_{\bar \F}^3,\Z_\ell(2))^{G_{\F}}=0.
$  
For a contradiction, assume that
$$
\alpha\in H^3(F_0 \CP_{\bar \F}^3,\Z_\ell(2))^{G_{\F}}
$$
is nonzero. 
Our assumption implies that $\alpha\in H^3(F_0 \CP_{\bar \F}^3,\Z_\ell(2))$ maps to zero in $H^3(F_0 \CP_{\bar \F}^3,\Z/\ell^r(2))$ for all $r$.
By the Bockstein sequence, $\alpha$ is $\ell^r$-divisible in  $H^3(F_0 \CP_{\bar \F}^3,\Z_\ell(2))$ for all $r$.
This implies that for any codimension one point $x\in \CP^3$, the residue
$$
\partial_x\alpha\in H^2(\bar x,\Z_\ell(1))
$$
is Galois-invariant and $\ell^r$-divisible for all $r$.
We aim to show that no such class exists.
Using prime to $\ell$ alterations (see \cite{IT}) we can without loss of generality assume that there is a smooth projective surface $D$ such that $\kappa(x)=\F(D)$.
Hence, 
$$
\partial_x\alpha\in H^2(F_0\bar D,\Z_\ell(1))
$$
is Galois-invariant and $\ell^r$-divisible for all $r$.
In particular, the natural map
$$
H^2(F_0\bar D,\Z_\ell(1))^{G_\F}\longrightarrow (H^2(F_0\bar D,\Z_\ell(1))\hat \otimes \Z_\ell )^{G_\F}
$$
has a nontrivial kernel. 
This contradicts the injectivity of \eqref{eq:map:prop:H^2-Gal invariant vs completion} in Proposition \ref{prop:H^2-Gal invariant vs completion}, which concludes the proof of the theorem.
\end{proof}

We are finally in the position to prove Theorem \ref{thm:main:Sq-Milnor-K-3}.

\begin{proof}[Proof of Theorem \ref{thm:main:Sq-Milnor-K-3}]
By \cite{morrow}, the Tate conjecture for divisors on smooth projective varieties over $\F=\F_q$ is equivalent to the analogous statement for smooth projective surfaces over $\F$. 
Hence, Theorem \ref{thm:main:Sq-Milnor-K-3} follows from Theorem \ref{thm:main:completion vs Tate for surfaces} together with the Galois-equivariant isomorphism
$$
H^3(F_0\CP^3_{\bar \F},\Z_\ell(3))\hat \otimes \Z_\ell \cong K^M_3(\bar \F(\CP^3))\hat \otimes \Z_\ell
$$
from Theorem \ref{thm:bloch-kato}.
\end{proof}

\section{$\ell$-adic completion of Milnor K-theory}

Recall from Section \ref{subsec:Milnor} our convention that Milnor K-theory of finite products of fields is the product of the Milnor K-theories of the individual fields.
In particular, if $X$ is a variety over $\F$ with base change $\bar X$, then for any codimension one point $x\in X^{(1)}$ in the smooth locus of $X$ there is a well-defined residue map $\del_{\bar x}:K^M_i(\bar \F(X))\to K^M_{i-1}(\kappa(\bar x)) $ (which takes into account that $\bar X$ as well as the base change $\bar x$ of the point $x$ may be reducible). 

\begin{lemma} \label{lem:del-x-1}
Let $X$ be a smooth projective variety over a finite field $\F$ and let $\ell$ be a prime invertible in $ \F$.
Let $\alpha\in K^M_i(\bar \F(X))$ be a class whose reduction modulo $\ell^r$ lies in the kernel of the residue map
$$
\del_{\bar x}:K^M_i(\bar \F(X))/\ell^r \longrightarrow K^M_{i-1}(\kappa(\bar x))/\ell^r
$$
for all $x\in X^{(1)}$ and all $r\geq 1$.
Assume that $i=1,2$ or that $ H_{nr}^i(\bar X,\mu_{\ell^r}^{\otimes i})=0$ for all $r\geq 1$.
Then $\alpha$ is $\ell^\infty$-divisible.
\end{lemma}
\begin{proof}
Let $\alpha\in K^M_i(\bar \F(X))$ such that $\del_{\bar x}\alpha$ is zero mod $\ell^r$ for all $x\in X^{(1)}$ and  all $r\geq 1$.
Then the image 
$$
[\alpha]_r\in K^M_i(\bar \F(X))/\ell^r\cong H^i(F_0 \bar X,\mu_{\ell^r}^{\otimes i})
$$
of $\alpha$  
is unramified, where the above isomorphism follows from \cite{Voe}.
If $H_{nr}^i(\bar X,\mu_{\ell^r}^{\otimes i})=0$, then $\alpha$ is zero modulo $\ell^r$ for all $r\geq 1$ and hence $\alpha$ is $\ell^\infty$-divisible, as we want.

Let now $i=1,2$. 
Then
$$ 
[\alpha]_r\in H_{nr}^i(\bar X,\mu_{\ell^r}^{\otimes i})=H^i(\bar X,\mu_{\ell^r}^{\otimes i})/N^1.
$$
(Note that $N^1H^1(\bar X,\mu_{\ell^r}^{\otimes i})=0$.)
The classes $[\alpha]_r$ for different values of $r$ form a projective system and so we get an element in the inverse limit: 
\begin{align} \label{eq:alpha_r-modN^1}
([\alpha]_r)\in \lim_{\substack{\longleftarrow\\ r}} (H^i(\bar X,\mu_{\ell^r}^{\otimes i})/N^1) .
\end{align}

We claim that 
\begin{align}\label{eq:limN^1}
\lim_{\substack{\longleftarrow\\ r}} (H^i(\bar  X,\mu_{\ell^r}^{\otimes i})/N^1) \cong H^i(\bar X,\Z_\ell(i))/N^1H^i(\bar X,\Z_\ell(i))  .
\end{align}
To see this, consider the short exact sequence
$$
0\longrightarrow N^1  H^i(\bar X,\mu_{\ell^r}^{\otimes i}) \longrightarrow H^i(\bar X,\mu_{\ell^r}^{\otimes i})\longrightarrow H^i(\bar X,\mu_{\ell^r}^{\otimes i})/N^1\longrightarrow 0.
$$
Since $N^1H^i(\bar X,\mu_{\ell^r}^{\otimes i})\subset H^i(\bar X,\mu_{\ell^r}^{\otimes i})$ is finite, the Mittag--Leffler condition holds and we get
$$
\lim_{\substack{\longleftarrow\\ r}}  (H^i(\bar X,\mu_{\ell^r}^{\otimes i})/N^1)\cong  H^i(\bar X,\Z_\ell(i)) / \lim_{\substack{\longleftarrow\\ r}} N^1  H^i(\bar X,\mu_{\ell^r}^{\otimes i}) .
$$
Recall that $N^1  H^i(\bar X,\mu_{\ell^r}^{\otimes i})$ is zero for $i=1$ and it coincides with the subgroup of algebraic classes for $i=2$ (because $\bar X$ is defined over $\bar \F$ which contains all roots of unity and so the Tate twists can be ignored). 
This implies 
$$
\lim_{\substack{\longleftarrow\\ r}} N^1  H^i(\bar X,\mu_{\ell^r}^{\otimes i})=N^1H^i(\bar X,\Z_\ell(i)),
$$
which concludes the proof of \eqref{eq:limN^1}. 

Via the isomorphism \eqref{eq:limN^1}, \eqref{eq:alpha_r-modN^1} yields a class
$$
([\alpha]_r)\in  H^i(\bar X,\Z_\ell(i))/N^1H^i(\bar X,\Z_\ell(i)) . 
$$
On the other hand, $\alpha_r$ is obtained via reduction modulo $\ell^r$ of $\alpha\in K^M_i(\bar \F(X))$.
Since the Frobenius action on $ K^M_i(\bar \F(X))$ has finite orbits, it follows that $([\alpha]_r)$ is fixed by some power of $\Frob_q$. 
Hence the class vanishes, because $H^i(\bar X,\Z_\ell(i))/N^1$ is torsion-free (see Corollary \ref{cor:torsion-free}) and the corresponding $\Q_\ell$-Galois module has weight $i-2i\neq 0$ (see \cite{deligne-weil}) hence no nontrivial class is fixed by a power of the Frobenius action.
It follows that $\alpha$ is $\ell^\infty$-divisible, as we want. 
\end{proof}

\begin{theorem} \label{thm:K_n-completed-countable-free}
Let $X$ be a  variety over  a finite field $\F$ and let $\ell$ be a prime invertible in $\F$.
Let $n$ be a positive integer such that one of the following holds:
\begin{enumerate}
\item $n=2$;
\item $n=3$ and $H^3_{nr}(\bar X,\mu_{\ell^r}^{\otimes 3})=0$ for all $r$ (e.g.\ $\bar X$ is rational).
\end{enumerate}
Then there is a $G_\F$-module $M$, whose underlying $\Z$-module is free of countable rank, such that there is a $G_\F$-equivariant isomorphism between the $\ell$-adic completions:
$$
K^M_n(\bar \F(X))\hat \otimes \Z_\ell=M\hat \otimes \Z_\ell .
$$
Moreover, $M$ is a $G_\F$-submodule of a countable sum of permutation $G_\F$-modules.
\end{theorem}
\begin{proof}
By the Bloch--Kato conjecture, proven by Voevodsky \cite{Voe}, we have a canonical isomorphism
$$
K^M_n(\bar \F(X))\hat \otimes \Z_\ell\stackrel{\cong}\longrightarrow \lim_{\substack{\longleftarrow\\ r}}  H^n(F_0\bar X,\mu_{\ell^r}^{\otimes n}) .
$$
Since affine varieties over $\bar \F$ have no cohomology in degrees larger than their dimensions, $H^n(F_0\bar X,\mu_{\ell^r}^{\otimes n})$ vanishes for  $n>\dim X$ and so does the above completion.
We may thus from now on assume that $n\leq \dim X$. 

Let
$\mathcal S^{(n)} :=\mathcal S^{(n)}(X)$ be the set of all tuples (up to isomorphism)
$$
\vec{x}:= (x_0,x_1,x_2,\dots ,x_n) ,
$$ 
consisting of the following inductive data:
\begin{itemize}
\item $x_0\in X^{(0)}$ is the generic point of $X$; 
\item if $x_i$ is defined, we choose a 
projective variety $C(x_i)$ whose function field is $\kappa(x_i)$ together with a prime to $\ell$ alteration $\tau(x_i): D(x_i)\to C(x_i)$ (see \cite{IT});
\item  if $x_i$, $C(x_i)$, $D(x_i)$ and $\tau(x_i)$ are defined, then $x_{i+1}\in D(x_i)^{(1)}$ is a codimension one point. 
\end{itemize} 
Each tuple $\vec{x}$ consists of the data of $(x_i,C(x_i),D(x_i),\tau(x_i))$ with $i=1,\dots ,n$; we drop $C(x_i),D(x_i),\tau(x_i)$ from the notation for convenience.
We emphasize that $\mathcal S^{(n)}$ is the set of all such tuples up to isomorphism, where the notion of isomorphism is the obvious one.
In particular,  
all projective varieties $C(x_i)$ with function field $\kappa(x_i)$, all prime to $\ell$ alterations $\tau(x_i): D(x_i)\to C(x_i)$  and all codimension one points $x_{i+1}\in D(x_i)^{(1)}$ appear. 
For the arguments below to work it would be possible to choose one $C(x_i)$, $D(x_i)$, and $\tau(x_i)$ for each $x_i$, but it is important that all $x_{i+1}\in D(x_i)^{(1)}$ appear.
(We choose the above canonical definition as it yields compatibilities of the sets $\mathcal S^{(n)}(X)$ for different $X$, as used for instance in the proof of Lemma \ref{lem:del-x-2} below.)

Since $\bar \F$ is countable, so is the set $\mathcal S^{(n)}(X)$.

Note that each $x_i$ is defined over $\F$; its base change $\bar x_i$ to $\bar \F$ will be a finite Galois orbit of points.
For each $\vec{x}\in \mathcal S^{(n)}$ and each $i$, there is a natural residue map
$$
\del_{i}(\vec{x}):
K^M_j(\kappa(\bar x_i))\longrightarrow K^M_{j-1}(\kappa(\bar x_{i+1})) ,
$$
given by the composition
$$
K^M_j(\kappa(\bar x_i))=K^M_j(\bar \F(C(x_i))) \stackrel{\tau(x_{i})^\ast}\longrightarrow K^M_j(\bar \F(D(x_i))) \stackrel{\del_{\bar x_{i+1}}}\longrightarrow K^M_{j-1}(\kappa(\bar x_{i+1})) .
$$ 
For each $\vec{x}\in \mathcal S^{(n)}$, we may then define a map
\begin{align} \label{del-vec-x}
\del_{\vec{x}}:K^M_n(\bar \F(X))\longrightarrow K^M_{0}(\kappa(\bar x_{n}))
\end{align}
via the composition
$$
K^M_n(\bar \F(X))=K^M_{n}(\kappa(\bar x_0)) \stackrel{\del_{0}(\vec{x})}\longrightarrow K^M_{n-1}(\kappa(\bar x_1)) \stackrel{\del_{1}(\vec{x})}\longrightarrow K^M_{n-2}(\kappa(\bar x_2)) \stackrel{\del_{2}(\vec{x})}\longrightarrow \dots 
\stackrel{\del_{n-1}(\vec{x})}\longrightarrow K^M_{0}(\kappa(\bar x_{n}))  .
$$
Note that $K^M_{0}(\kappa(\bar x_{n})) $ is a free abelian group whose (finite) rank is given by the number of points in the scheme $\bar x_n$, i.e.\ the length of the Galois orbit over $\bar \F$ that corresponds to $x_n$.
The Galois action is induced by the natural action on this orbit and so  $K^M_{0}(\kappa(\bar x_{n})) $ is a permutation $G_\F$-module.

\begin{lemma} \label{lem:del-x-2} 
Let $n\leq 3$ be a natural number.
If $n=3$, we assume that $H^3_{nr}(\bar X,\mu_{\ell^r}^{\otimes 3})=0$ for all $r$.
Then the kernel of the map 
$$
\oplus_{\vec{x}\in \mathcal S^{(n)}}  \del_{\vec{x}}:K^M_n(\bar \F(X))\longrightarrow \bigoplus_{ \vec{x}\in \mathcal S^{(n)}} K^M_{0}(\kappa(\bar x_{n}))
$$
is $\ell$-divisible.  
\end{lemma}
\begin{proof}
We recall that an abelian group $A$ is $\ell$-divisible if $A/\ell=0$.
In other words, $A$ is $\ell$-divisible if any element of $A$ is $\ell$-divisible (and hence in fact, any element is $\ell^\infty$-divisible).

We prove the lemma by induction on $n$.
For convenience of notation, we denote the generic point of $D(x_i)$ by $x_i'$.
Then $\tau(x_i)$ induces a field extension $\kappa(x_i)\subset \kappa(x_i')$ whose degree is finite and coprime to $\ell$.
If we base change this to $\bar \F$, we get a finite morphism of zero-dimensional schemes $\Spec \kappa(\bar x_i')\to \Spec \kappa(\bar x_i)$ whose degree on the various connected components is constant and coprime to $\ell$.

For $n=1$, the statement boils down to showing that any an invertible regular function $\alpha\in \kappa(\bar x_0)^\ast$ on $\Spec \kappa(\bar x_0)$,  which lies in the kernel of
$$
\kappa(\bar x_0)^\ast\longrightarrow \kappa(\bar x'_0)^\ast \stackrel{\del_{\bar x_1}}\longrightarrow  K^M_{0}(\kappa(\bar x_{1}))
$$
for all $\vec x\in \mathcal S^{(1)}$,  is $\ell$-divisible. 
To prove this,  let $\alpha\in  \kappa(\bar x_0)^\ast$ be as above.
By Lemma \ref{lem:del-x-1}, the image of $\alpha$ in $\kappa(\bar x'_0)^\ast $ is $\ell$-divisible.
It then follows from a norm argument that $\alpha$ is $\ell$-divisible,  as we want, where we use that the alterations chosen above have degree coprime to $\ell$.

Let now $n>1$ and let $\alpha\in K^M_n(\bar \F(X))$ be in the kernel of $\del_{\vec x}$ for all $\vec x\in \mathcal S^{(n)}$.
Via a norm argument along the morphism $\tau(x_0):D(x_0)\to C(x_0)$, where $C(x_0)$ is birational to $X$, as above, we reduce to the case where $X$ is smooth projective.
For each $x\in X^{(1)}$, $\del_{\bar x}\alpha\in K^M_{n-1}(\kappa(\bar x))$ lies by construction in the kernel of $\del_{\vec{y}}$ for all $\vec{y}\in \mathcal S^{(n-1)}(\overline{\{x\}})$.  
We conclude by induction that $\del_{\bar x}\alpha$ is $\ell$-divisible for all $x\in X^{(1)}$. 
If either $n=2$, or $n=3$ and $H^3_{nr}(\bar X,\mu_{\ell^r}^{\otimes 3})=0$ for all $r$, then we conclude from Lemma \ref{lem:del-x-1} that $\alpha$ is $\ell$-divisible, as we want. 
This concludes the proof of the lemma. 
\end{proof}

Consider the $G_\F$-module
$$
N:= \bigoplus_{ \vec{x}\in \mathcal S^{(n)}} K^M_{0}(\kappa(\bar x_{n})),
$$
where $\mathcal S^{(n)}=\mathcal S^{(n)}(X)$.
Each $K^M_{0}(\kappa(\bar x_{n}))$ is a finite rank permutation $G_\F$-module.
Since $\mathcal S^{(n)}$ is countable, $N$ is a free $\Z$-module of countable rank, given as a countable direct sum of  permutation $G_\F$-modules.
We then let
$$
M:=\im \left( \oplus_{\vec{x}\in \mathcal S^{(n)} }  \del_{\vec{x}}:K^M_n(\bar \F(X))\longrightarrow  \bigoplus_{ \vec{x}\in \mathcal S^{(n)}} K^M_{0}(\kappa(\bar x_{n})) \right)\subset N ,
$$
which is a $G_\F$-submodule of $N$.
Since $N$ is a free $\Z$-module of countable rank, so is $M$.

By Lemma \ref{lem:del-x-2}, there is a short exact sequence
$$
0\longrightarrow D\longrightarrow K^M_n(\bar \F(X))\longrightarrow M\longrightarrow 0,
$$
where $D$ is $\ell$-divisible.
It follows that the above sequence induces $G_\F$-equivariant isomorphisms
$$
K^M_n(\bar \F(X))/\ell^r \cong  M/\ell^r
$$
for all $r$ and hence a $G_\F$-equivariant isomorphism
$$
K^M_n(\bar \F(X))\hat \otimes \Z_\ell \cong M\hat \otimes \Z_\ell.
$$
This concludes the proof of the theorem.
\end{proof}

\section{$\ell$-adic completions of tensor products}

Let $X$ be a geometrically irreducible variety over a finite field $\F=\F_q$.
By definition, there is a natural surjection 
$$
(\bar \F(X)^\ast)^{\otimes i}\twoheadrightarrow K^M_{i}(\bar \F(X)) .
$$
By \cite[\href{https://stacks.math.columbia.edu/tag/00M9}{Tag 00M9}]{stacks-project}, this induces a surjection on $\ell$-adic completions and hence we get a surjection
$$
((\bar \F(X)^\ast)^{\otimes i})\hat \otimes \Q_\ell \twoheadrightarrow K^M_{i}(\bar \F(X))\hat \otimes \Q_\ell.
$$
This surjection is equivariant with respect to the Frobenius action and so it is natural to wonder about the Frobenius eigenvalues on $((\bar \F(X)^\ast)^{\otimes i})\hat \otimes \Q_\ell $.

\begin{proposition}\label{prop:tensor product}
Let $X$ be a geometrically irreducible variety over a finite field $\F=\F_q$.
Assume that $\Pic \bar X$ is a free abelian group.
Then all eigenvalues of the $\Frob_q$-action on $(\bar \F(X)^\ast)^{\otimes i}\hat \otimes \Q_\ell$ are roots of unity.
\end{proposition}

\begin{proof}
Since $\Frob_{q^m}=\Frob_q^m$, it suffices to prove the statement up to replacing $\F=\F_q$ by a finite extension.
In particular, we may assume that $\Frob_q$ acts trivially on the free abelian group $\Pic \bar X$.
There is a natural Galois-equivariant exact sequence
$$
0\longrightarrow \bar \F(X)^\ast/\bar \F^\ast\stackrel{\del} \longrightarrow \bigoplus_{y\in \bar X^{(1)}} [y]\Z\longrightarrow \Pic \bar X\longrightarrow 0.
$$
The group $\bigoplus_{y\in \bar X^{(1)}} [y]\Z$ is a direct sum of permutation modules given by the Galois orbits of codimension one points on $\bar X$.
Up to replacing $\F$ by a finite extension, we may assume that a basis of $\Pic \bar X$ can be represented by a linear combination of irreducible subvarieties that are all defined over $\F$.
Under this assumption the surjection in the above short exact sequence admits a $G_\F$-equivariant section and so the sequence splits as a sequence of $G_\F$-modules.
It follows that $\bar \F(X)^\ast/\bar \F^\ast$, viewed as a Galois module, is a direct sum of permutation modules.
Hence the same holds for the tensor product
$$
(\bar \F(X)^\ast/\bar \F^\ast)^{\otimes i}.
$$
Since $\bar \F$ is algebraically closed,  $\bar \F^\ast$ is a divisible group (in fact a divisible torsion group because $\F$ is finite).
It follows that the natural map
$$
(\bar \F(X)^\ast)^{\otimes i}\hat \otimes R \longrightarrow
(\bar \F(X)^\ast/\bar \F^\ast)^{\otimes i}\hat \otimes R
$$
is an isomorphism for $R\in \{\Z_\ell,\Q_\ell\}$ and so the proposition follows from Lemma \ref{lem:permutation} below. 
\end{proof}


\begin{lemma} \label{lem:permutation}
Let $G=\hat \Z$ and let $N$ be a $G$-module.
Assume that $N\cong \oplus_{i\in I} P_i$ decomposes into a direct sum of permutation $G$-modules $P_i$,.
Then any eigenvalue of the action
$$
F:N\hat \otimes \Q_\ell\longrightarrow N\hat \otimes \Q_\ell
$$
induced by the element $1\in G$ is a root of unity.
\end{lemma}
\begin{proof}
An element in $(\oplus_{i\in I} P_i)\hat \otimes \Z_\ell $ is given by an inverse system $(\sum_i \alpha_{i,r})_r$, where $\alpha_{i,r}\in P_i/\ell^r$ and for each $r$, all but finitely many $\alpha_{i,r}$ are zero.
We thus get a natural map
$$
\left(\bigoplus_{i\in I} P_i \right)\hat \otimes \Z_\ell \hookrightarrow \prod_{i\in I} P_i\hat \otimes \Z_\ell,\ \ \ (\sum_i \alpha_{i,r})_r\mapsto ((\alpha_{i,r})_r)_i. 
$$
Since each $P_i$ is finitely generated, we have $P_i\hat \otimes \Z_\ell=P_i \otimes \Z_\ell$.
It follows that the above map is injective, because $(\alpha_{i,r})_r=0$ for all $i$ implies that $\alpha_{i,r}=0$ for all $i,r$.
Altogether we have established a $G$-equivariant embedding
$$
\left(\bigoplus_{i\in I} P_i \right)\hat \otimes \Z_\ell \hookrightarrow \prod_{i\in I} P_i \otimes \Z_\ell .
$$
The eigenvalues of the natural action of $F$ on the right hand side are the eigenvalues of the action on $P_i \otimes \Z_\ell $, which are roots of unity because $P_i$ is a permutation module.
This concludes the proof of the lemma.
\end{proof}

\section{Weights for completed Milnor K-theory}
Let  $Y$ be a projective variety over a finite field $\F=\F_q$ with $q$ elements and let $\Frob_q$ denote the arithmetic Frobenius relative to $\F$.
We say that $\lambda\in Z_\ell$ is an eigenvalue of the $\Frob_q$-action on $ K^M_3(\bar \F(Y))\hat \otimes \Q_\ell$ if there is a nonzero element 
$$
\alpha\in K^M_3(\bar \F(Y))\hat \otimes \Q_\ell 
$$
with $\Frob_q(\alpha)=\lambda \alpha$.
We say that $\lambda$ has absolute value $v\in \R$ if for all embeddings $\bar \Q_\ell \hookrightarrow \C$, the absolute value of $\lambda$ is $v$.

\begin{proposition} \label{prop:weights-for-completed-K-theory}
Let  $Y$ be a projective variety over a finite field $\F=\F_q$ with $q$ elements and let $\Frob_q$ denote the arithmetic Frobenius relative to $\F$.
Let $\ell$ be a prime invertible in $\F$ and assume that $H_{nr}^3(\bar Y,\mu_{\ell^r}^{\otimes 3})=0$ for all $r\geq 1$. 
Then any eigenvalue $\lambda\in \Z_\ell$ of the  $\Frob_q$-action on $K^M_3(\bar \F(Y))\hat \otimes \Q_\ell $ has absolute value $|\lambda|\in \{1,q^{1/2},q\} $. 
\end{proposition}
\begin{proof}
By Theorem \ref{thm:bloch-kato}, there is a canonical Galois-equivariant isomorphism
\begin{align} \label{eq:Sq-Galois-equivariant-iso}
K^M_3(\bar \F(Y))\hat \otimes \Z_\ell\cong H^3(F_0\bar Y,\Z_\ell(3))\hat \otimes \Z_\ell .
\end{align} 
Let $\alpha=(\alpha_r)\in  H^3(F_0\bar Y,\Z_\ell(3))\hat \otimes \Z_\ell$ be non-torsion and assume that $\Frob_q^\ast \alpha=\lambda \alpha$ for some $\lambda\in \Z_\ell$ with $$
|\lambda|\notin  \{1,q^{1/2},q\}.
$$
Since $H_{nr}^3(\bar Y,\mu_{\ell^r}^{\otimes 3})=0$ for all $r\geq 1$, there must be a codimension one point $x\in Y^{(1)}$, such that the residue
$$
(\partial_x\alpha_r)\in H^2(\bar x,\Z_\ell(2))\hat \otimes \Z_\ell
$$
is nonzero, hence a $\lambda$-eigenvector for the $\Frob_q^\ast$-action.
Using prime to $\ell$ alterations, we can without loss of generality assume that there is a smooth projective $\F$-variety $X$ with $\kappa(x)=\bar \F(X)$.
We then get a nonzero class $\beta=(\partial_x\alpha_r)\in H^2(\bar \F(X),\Z_\ell(2))\hat \otimes \Z_\ell$ with $\Frob_q^\ast \beta=\lambda\beta$.
Let $y\in X^{(1)}$ be a codimension 1-point.
The residue
$$
\partial_y\beta\in H^1(\bar y,\Z_\ell(1))\hat \otimes \Z_\ell
$$
is either zero or a $\lambda$-eigenvector of $\Frob_q^\ast$.
Since $|\lambda|\neq 1$,  another residue computation shows that $\partial_y\beta$ is unramified and hence lifts to
$$
H^1_{nr}(\bar y,\Z_\ell(1))\hat \otimes \Z_\ell.
$$
The above group is a finite rank $\Z_\ell$-module  all of whose $\Frob_q^\ast$-eigenvalues have absolute value $q^{1/2}$ by \cite{deligne-weil} (use prime to $\ell$ alterations \cite{IT}). 
Since $|\lambda|\neq q^{1/2}$, we deduce $\partial_y\beta=0$.
It follows that 
$$
\beta_r= \partial_x\alpha_r\in H^2_{nr}(\bar X,\Z/\ell^r(2))= H^2(\bar X,\Z/\ell^r(2))/N^1.
$$
These classes are compatible for various $r$ and hence we get in the inverse limit a class
$$
(\partial_x\alpha_r)\in (H^2(\bar X,\Z_\ell(2))/N^1)\hat \otimes \Z_\ell=H^2(\bar X,\Z_\ell(2))/N^1,
$$
where we used that $N^1H^2$ is generated by algebraic classes and hence
$$
N^1H^2(\bar X,\Z_\ell(2))\otimes \Z/\ell^r\cong N^1H^2(\bar X,\Z/\ell^r(2))
$$
for all $r$.
Since $\bar \F$ contains all $\ell$-th roots of unity,
 $H^2(\bar X,\Z_\ell(2))/N^1$ is torsion-free, see Corollary \ref{cor:torsion-free}.
Since $(\partial_x\alpha_r) $ is nonzero, it follows that $\lambda$ is an eigenvalue of the $\Frob_q^\ast$-action on $ H^2(\bar X,\Z_\ell(2))$. 
Since $\Frob_q$ is the arithmetic Frobenius, the Weil conjectures (see \cite{deligne-weil}) imply $|\lambda|=q$, which contradicts the assumption $|\lambda|\notin  \{1,q^{1/2},q\}$.
\end{proof}

\section{Frobenius Eigenvalues on completed Milnor K-theory of $\CP^3$}

\begin{theorem} \label{thm:Sq}
Let $\F=\F_q$ be a finite field with $q$ elements.
Let $\Frob_q $ be the arithmetic Frobenius relative to $\F_q$ and consider the set of integral eigenvalues of $\Frob_q^\ast$ on $K^M_3(\bar \F(\CP^3))\hat \otimes \Z_\ell $: 
$$
S_q:=\{\lambda\in \Z \mid \text{ there is a nonzero class $\alpha\in K^M_3(\bar \F(\CP^3))\hat \otimes \Z_\ell $ with $\Frob_q^\ast \alpha=\lambda\alpha$.} \}
$$ 
Then:
\begin{enumerate}
\item \label{item1:weights-Milnor-K-3}
$S_q\subset \{\pm1,\pm q^{1/2},\pm q\}$; 
\item If $q=p^{4m}$, then $\{\pm1, \pm q^{1/2}\}\subset S_q$. 
\label{item2:eigenvalues-Milnor-K-3} 
\end{enumerate}  
\end{theorem}

We will use the following example from \cite{katz-sarnak} for the proof of item \eqref{item2:eigenvalues-Milnor-K-3} in Theorem \ref{thm:Sq}.

\begin{lemma}\label{lem:curve:katz-sarnak}
Let $q=p^m$ be a power of $p$ and consider the Fermat curve $X\subset \CP^2$, given by $x^{q+1}+y^{q+1}=z^{q+1}$. 
Then  $\Frob_{q^4}$ acts by multiplication with $q^{2}$ on $H^1(\bar X,\Z_\ell(1))$.
\end{lemma}
\begin{proof} 
As a consequence of a point count and the Weil bound, the geometric Frobenius relative to $\F_{q^2}$ acts via multiplication by $-q$  on $H^1(\bar X,\Z_\ell)$,  see  \cite[p.\ 8]{katz-sarnak}, and hence via multiplication by $-q^{-1}$ on $H^1(\bar X,\Z_\ell(1))$.
Since the action of the geometric Frobenius relative to $\F_{q^2}$ agrees with the inverse of the action of the arithmetic Frobenius $\Frob_{q^2}$, we find that 
 $\Frob_{q^2}$ acts via $-q$ and hence $\Frob_{q^4}$ acts via $q^2$ on $H^1(\bar X,\Z_\ell(1))$,  as we want.
\end{proof}

For the proof of Theorem \ref{thm:Sq} it will be useful to use Borel--Moore cohomology, which for an equi-dimensional algebraic scheme $X$ over a field $k$ is defined by 
$$
H^{i}_{BM}(X, A(n)):=H^{i-2d_X} (X_{\proet},\pi_X^!A(n-d_X)), 
$$
where $A\in \{\Z/\ell^r,\Z_\ell,\Q_\ell\}$ for a prime $\ell$ invertible in $k$, $d_X=\dim X$,  $\pi_X:X\to \Spec k$ denotes the structure map and $X_\proet$ is the pro-\'etale site of $X$, see \cite[Section 4 and Proposition 6.6]{Sch-refined}.
If $k$ is a finite field or the algebraic closure of a finite field, then the pro-\'etale site may be replaced by the \'etale site in the above formula, as in this case the above groups are finitely generated $A$-modules. 
Basic properties are listed in \cite[\S 4, \S 6]{Sch-refined} and in \cite[\S 2.4]{Balkan-Sch}.
Most importantly,  there are proper pushforwards \cite[Proposition 6.6, P1]{Sch-refined}, there is a Gysin sequence \cite[Proposition 6.6, P2]{Sch-refined}, and Borel--Moore cohomology agrees with continuous $\ell$-adic cohomology if $X$ is smooth and equi-dimensional \cite[Lemma 6.5]{Sch-refined}.
If $X$ is defined over a finite field $\F$, then $G_\F$ acts naturally on $H^{i}_{BM}(\bar X, A(n))$, see e.g.\ \cite[\S 3.2]{Balkan-Sch}.

\begin{lemma} \label{lem:nodal-curve-times-Fermat}
Let $X$ be as in Lemma \ref{lem:curve:katz-sarnak} and let $Y\subset \CP^2$ be a nodal cubic curve over $\F=\F_q$ with one node  $y_0\in Y$.
Then there is a class $\epsilon\in H^2_{BM}(\bar X\times \bar Y,\Z_\ell(2))$ with the following properties:
the image of $\epsilon$ in $H^2 (F_0(\bar X\times \bar Y),\Z_\ell(2))$ is nonzero and an eigenvector for the $\Frob_{q^4}$-action with eigenvalue $q^{2}$.
\end{lemma}
\begin{proof}
The normalization of $Y$ is $\CP^1$ and we may assume that the preimage of the singular point $y_0\in Y$ corresponds to $0,\infty\in \CP^1$.
In particular,   $Y^{\operatorname{reg}}=\mathbb G_m$.
The Gysin sequence yields a short exact sequence
\begin{align} \label{eq:Gysin-les-1}
H^{-1}(y_0,\Z_\ell(0))=
0\longrightarrow H^1_{BM}(\bar Y,\Z_\ell(1))\longrightarrow H^1(\mathbb G_{m,\bar \F},\Z_\ell(1)) \stackrel{\del}\longrightarrow H^0(y_0,\Z_\ell(0)) .
\end{align}
Applying the Gysin sequence to $\mathbb G_{m,\bar \F}\subset \CP^1_{\bar \F}$, we see that 
the group $H^1(\mathbb G_{m,\bar \F},\Z_\ell(1))\cong \Z_\ell$ is generated by a class that has opposite residues at $0$ and $\infty$.
By functoriality of the Gysin sequence, the residue map in \eqref{eq:Gysin-les-1} is the sum of the residues at 0 and infinity, hence it is zero.
It follows that $H^1_{BM}(\bar Y,\Z_\ell(1))\cong \Z_\ell$ is a free $\Z_\ell$-module of rank 1 with trivial Galois action.

Consider the singular surface
$
X\times Y.
$
The Gysin sequence yields an exact sequence
$$
H_{BM}^0(\bar X\times\{y_0\},\Z_\ell(1))\longrightarrow 
H^2_{BM}(\bar X\times \bar Y,\Z_\ell(2))\longrightarrow H^2_{BM}(\bar X\times \bar Y^{\reg},\Z_\ell(2))\longrightarrow H_{BM}^1(\bar X\times\{y_0\},\Z_\ell(1)) .
$$
Since $X\times Y^{\reg}$ is smooth, we find that
$$
H^2_{BM}(\bar X\times \bar Y^{\reg},\Z_\ell(2))=H^2 (\bar X\times \bar Y^{\reg},\Z_\ell(2)).
$$
Since the $\ell$-adic cohomology of $\bar X$ and $\bar Y^{\reg}$ is torsion-free (both are curves over an algebraically closed field),  the K\"unneth formula yields a Galois-equivariant isomorphism
$$
H^2 (\bar X\times \bar Y^{\reg},\Z_\ell(2))\cong H^1 (\bar X,\Z_\ell(1))\otimes H^1(\bar Y^{\reg},\Z_\ell(1)) ,
$$
see \cite[Theorem 22.4]{milne-notes}.

Let $\alpha\in  H^1 (\bar X,\Z_\ell(1))$ and $\beta\in  H^1(\bar Y^{\reg},\Z_\ell(1))$ be nonzero.
By Lemma \ref{lem:curve:katz-sarnak},  $\Frob_{q^4}(\alpha)=q^{2}\alpha$.
Moreover, $\beta$ is Galois-invariant because $H^1(\bar Y^{\reg},\Z_\ell(1))=\Z_\ell$ with trivial Galois action.
Via the above K\"unneth formula, we get a nontrivial class
$$
\epsilon:=\alpha\otimes \beta\in H^2_{BM}(\bar X\times \bar Y^{\reg},\Z_\ell(2))
$$
with $\Frob_{q^4}(\epsilon)=q^{2}\epsilon$.
We have seen above that $\beta$ has trivial residues on $\bar Y$.
It thus follows from the above Gysin sequence that $\epsilon$ lifts to a class in $H^2_{BM}(\bar X\times \bar Y,\Z_\ell(2))$ that we denote by the same symbol:
$$
\epsilon\in H^2_{BM}(\bar X\times \bar Y,\Z_\ell(2)).
$$

It remains to show that the restriction of $\epsilon $ to $F_0(\bar X\times \bar Y)$ is nonzero.
Since $Y$ is rational,  we have $F_0(\bar X\times \bar Y)= F_0(\bar X\times \CP_{\bar \F}^1)$.
There is a residue map
$$
\del_{\bar X\times \{0\}}:H^2 (F_0(\bar X\times \bar Y),\Z_\ell(2))\longrightarrow H^1( F_0\bar X,\Z_\ell(1))
$$
that is associated to the divisor $\bar X\times \{0\}$ on $ \bar X\times \CP_{\bar \F}^1$.
The image of $\epsilon$ under this map is given by
$$
\del_{\bar X\times \{0\}} (\epsilon)=-\partial_{0}(\beta)\cdot \alpha\in H^1(F_0\bar X,\Z_\ell(1)),
$$
see \cite[Lemma 2.4]{Sch-survey}.
The above class is nonzero because $\partial_0\beta$ is nonzero and $H^1(\bar X,\Z_\ell(1))\to H^1(F_0\bar X,\Z_\ell(1)) $ is injective by the Gysin sequence.
Hence, the image of $\epsilon$ in $H^2 (F_0(\bar X\times \bar Y),\Z_\ell(2))$ is nonzero, as we want.
This concludes the proof of the lemma. 
\end{proof}

\begin{proof}[Proof of Theorem \ref{thm:Sq}] 
Item \eqref{item1:weights-Milnor-K-3} follows from Proposition \ref{prop:weights-for-completed-K-theory}.

To prove \eqref{item2:eigenvalues-Milnor-K-3},  
let $X$ and $Y$ be as in Lemmas \ref{lem:curve:katz-sarnak} and \ref{lem:nodal-curve-times-Fermat}.
We embed $X\times Y$ into some projective space and project generically to $\CP^3$ to obtain an integral subscheme $Z\subset \CP^3$ with a finite birational map $f:X\times Y\to Z$.
Since $f$ is proper, we can consider the class
$$
f_\ast \epsilon\in H^2_{BM}(Z,\Z_\ell(2)) ,
$$ 
where $\epsilon\in H_{BM}^2(\bar X\times \bar Y,\Z_\ell(2))$ is as in Lemma \ref{lem:nodal-curve-times-Fermat}.
Since $f$ is birational, the restriction of this class to $F_0\bar Z$ agrees with the restriction of $\epsilon$ to $F_0(\bar X\times \bar Y)$.
Let $U=\CP^3\setminus Z$ be the complement of $Z$.
There is a short exact sequence
$$
H^3_{BM}(\CP^3_{\bar \F}, \Z_\ell(3)) \longrightarrow H^3_{BM}(\bar U, \Z_\ell(3))\longrightarrow  H^2_{BM}(\bar Z,\Z_\ell(2))\longrightarrow H^4_{BM}(\CP^3_{\bar \F}, \Z_\ell(3)).
$$
The class $f_\ast \epsilon$ has to map to zero in $H^4_{BM}(\CP^3_{\bar \F}, \Z_\ell(3))\cong \Z_\ell(1)$ for weight reasons and so it lifts to a class 
$$
\alpha\in H^3_{BM}(\bar U, \Z_\ell(3)).
$$
Since $\alpha$ is defined on $\bar U$, the only possible residue of its restriction to $F_0\CP^3_{\bar \F}$ is given by the residue along the map
$$
\del_Z:H^3 (F_0 \CP^3_{\bar \F}, \Z_\ell(3))\longrightarrow H^2(F_0Z,\Z_\ell(2)) .
$$
By Lemma \ref{lem:nodal-curve-times-Fermat}, we have 
$$
\del_Z  \Frob_{q^4}^\ast \alpha=
\Frob_{q^4}^\ast \del_Z \alpha= q^{2} \del_Z \alpha.
$$
It follows that the class 
$$
\Frob_{q^4}^\ast\alpha-q^{2} \alpha\in H^3(F_0\CP^3_{\bar \F},\Z_\ell(3))
$$
is unramified. 
Hence the above class vanishes because $\CP^3_{\bar \F}$ has no unramified cohomology in positive degree (see e.g.\ \cite[Corollary 1.8(1) for $j=0$]{Sch-moving}).
It follows that $\Frob_{q^4}^\ast\alpha=q^{2} \alpha$.
Since $\alpha$ is nonzero we conclude that $q^{2}$ is an eigenvalue of $\Frob_{q^4}^\ast$ on $H^3 (F_0 \CP^3_{\bar \F}, \Z_\ell(3))$  for any $q=p^m$.

Replacing $q$ by $q^4$ shows that $q^{1/2}\in S_q$ whenever $q=p^{4m}$ is a fourth power.
Replacing in the above argument $Y$ by a union of two smooth Galois conjugate rational curves that meet in two points, we obtain via a similar argument that $-q^{1/2}\in S_q$ if $q=p^{4m}$ is a $4$-th power.

It remains to show that $\lambda:=\pm 1\in S_q$. 
To this end consider the function field $\bar \F(x)$ in one variable.
There is a non-constant rational function $\xi\in \bar \F (x)^\ast $ with $\Frob_q ( \xi)=\lambda \xi$, where $ \bar \F (x)^\ast=K_1^M(\bar \F(x))$ is written as an additive group.
(For instance, if $\lambda=-1$, then we can pick $\xi=f/g$ for linear polynomials $f$ and $g$ that are interchanged by Frobenius.)
We denote affine coordinates on $\A^3\subset \CP^3$ by $x,y,z$ and find that the class $\alpha:=(\xi,y,z)\in K^M_3(\bar \F(\CP^3))$ is nonzero (evaluate the residue at $z=0$ followed by the residue at $y=0$) and satisfies $\Frob_q^\ast \alpha=\lambda \alpha$.
The image of $\alpha$ in the $\ell$-adic completion $K^M_3(\bar \F(\CP^3))\hat \otimes \Z_\ell$ is nonzero by a similar residue computation as above.
This proves $\pm 1\in S_q$, which concludes the proof of  \eqref{item2:eigenvalues-Milnor-K-3}.
This concludes the proof of the theorem.
\end{proof}

\section{Proof of Theorem \ref{thm:basic-properties}}

\begin{proof}[Proof of Theorem \ref{thm:basic-properties}]
Item \eqref{item1:thm:basic-properties} follows from Theorem \ref{thm:K_n-completed-countable-free} together with the straightforward observation that $K^M_3 (\bar \F(\CP^3))\hat \otimes \Z_\ell$ is not a $\Z_\ell$-module of finite rank. 

We will prove items \eqref{item2:thm:basic-properties}  and \eqref{item3:thm:basic-properties} more generally for $K^M_n(\bar \F(X))$ in place of $K^M_3(\bar \F(\CP^3))$, where  $X$ denotes any geometrically irreducible variety over $\F=\F_q$.
Item \eqref{item2:thm:basic-properties} follows in this generality  from the fact that the $\Frob_q$-action has finite orbits on $K^M_n(\bar \F(X))$, hence on $K^M_n(\bar \F(X))\otimes \Q_\ell$, and so all eigenvalues are roots of unity.

To prove item \eqref{item3:thm:basic-properties}, 
assume that   $\Frob_q-q\cdot \id$ has a nontrivial kernel on $K^M_n(\bar \F(X))\hat \otimes \Q_p$.
This implies that there is a nontorsion class $\alpha\in K^M_n(\bar \F(X))\hat \otimes \Z_p$ such that $\Frob_q(\alpha)=q\cdot \alpha$.
To see that this is impossible, note that $\Frob_q$ acts via isomorphisms on  $K^M_n(\bar \F(X))$ and hence also via isomorphisms on the completion
$$
K^M_n(\bar \F(X))\hat \otimes \Z_p =\lim_{\substack{\longleftarrow\\ r}} K^M_n(\bar \F(X))/p^r.
$$
Let $\phi_q:K^M_n(\bar \F(X))\hat \otimes \Z_p\to K^M_n(\bar \F(X))\hat \otimes \Z_p$ be an inverse of $\Frob_q$.
Then $\Frob_q(\alpha)=q\cdot \alpha$ implies $\Frob_q^{\circ r}(\alpha)=q^r\cdot \alpha$ and hence $\alpha=q^r\cdot \phi_{q^{r}}^{\circ r}(\alpha)$ is $q^r$-divisible for all $r$.
This implies
$$
\alpha=0\in K^M_n(\bar \F(X))\hat \otimes \Z_p,
$$ 
which contradicts the assumption that $\alpha$ was nontorsion.

 Item \eqref{item4:thm:basic-properties} follows from Proposition \ref{prop:tensor product} and item \eqref{item5:thm:basic-properties} follows from Theorem \ref{thm:Sq}.
 This concludes the proof of the theorem.
\end{proof}

\section*{Acknowledgements} 
I am very grateful to the referee for his or her comments.
This project has received funding from the European Research Council (ERC) under the European Union's Horizon 2020 research and innovation programme under grant agreement No 948066 (ERC-StG RationAlgic).


\end{document}